\documentclass[12pt]{amsart}

\usepackage{amssymb}
\usepackage{bm}
\usepackage{graphicx}
\usepackage{psfrag}
\usepackage{color}
\usepackage{url}
\usepackage{algpseudocode}
\usepackage{fancyhdr}
\usepackage{xy}
\usepackage{caption}
\usepackage{subcaption}
\input xy
\xyoption{all}

\newtheorem*{maintheorem*}{Main Theorem}
\newtheorem{theorem}{Theorem}[section]
\newtheorem{prop}[theorem]{Proposition}

\newtheorem{question}[theorem]{Question}
\newtheorem{lemma}[theorem]{Lemma}
\newtheorem{cor}[theorem]{Corollary}
\theoremstyle{definition}
\newtheorem{definition}[theorem]{Definition}
\newtheorem{example}[theorem]{Example}
\newtheorem{remark}[theorem]{Remark}
\numberwithin{equation}{section}

\newcommand{\cc}{\mathbb{C}}
\newcommand{\rr}{\mathbb{R}}
\newcommand{\nn}{\mathbb{N}}
\newcommand{\zz}{\mathbb{Z}}
\newcommand{\bb}{\mathcal{B}}
\newcommand{\ii}{\mathcal{I}}
\newcommand{\T}{\mathcal{T}}

\newcommand{\mfb}{\mathfrak{b}}
\newcommand{\mfc}{\mathfrak{c}}
\newcommand{\mfd}{\mathfrak{d}}
\newcommand{\mff}{\mathfrak{f}}

\newcommand{\mfi}{\mathfrak{i}}
\newcommand{\mfr}{\mathfrak{r}}
\newcommand{\mfs}{\mathfrak{s}}
\newcommand{\mft}{\mathfrak{t}}
\newcommand{\mfu}{\mathfrak{u}}


\begin{document}
	
	\mbox{}
	\title{Tilings and matroids on \\ regular subdivisions of a triangle}
	\author{Felix Gotti}
	%\address{Department of Mathematics\\UC Berkeley\\Berkeley, CA 94720}
	%\email{felixgotti@berkeley.edu}
	\author{Harold Polo}
	%\address{Department of Mathematics\\UC Berkeley\\Berkeley, CA 94720}
	%\email{haroldpolo@berkeley.edu}
	
	\address{Department of Mathematics\\UC Berkeley\\Berkeley, CA 94720}
	\email{felixgotti@berkeley.edu, haroldpolo@berkeley.edu}
	%\urladdr{www.felixgotti.com}
	
	\date{\today}
	
	\begin{abstract}
		In this paper we investigate a family of matroids introduced by Ardila and Billey to study one-dimensional intersections of complete flag arrangements of $\cc^n$. The set of lattice points $P_n$ inside the equilateral triangle $S_n$ obtained by intersecting the nonnegative cone of $\rr^3$ with the affine hyperplane $x_1 + x_2 + x_3 = n-1$ is the ground set of a matroid $\T_n$ whose independent sets are the subsets $S$ of $P_n$ satisfying that $|S \cap P| \le k$ for each translation $P$ of the set $P_k$. Here we study the structure of the matroids $\T_n$ in connection with tilings of $S_n$ into unit triangles, rhombi, and trapezoids. First, we characterize the independent sets of $\T_n$, extending a characterization of the bases of $\T_n$ already given by Ardila and Billey. Then we explore the connection between the rank function of $\T_n$ and the tilings of $S_n$ into unit triangles and rhombi. Then we provide a tiling characterization of the circuits of $\T_n$. We conclude with a geometric characterization of the flats of $\T_n$.
		%Finally, we use tilings to prove that the matroids $\T_{n,3}$ are connected.
		%for every $d \ge 3$ and illustrate how rhombic tilings can be used to argue the connectedness of $\T_{n,3}$. \\
	\end{abstract}

\maketitle

%\tableofcontents

\section{Introduction}

%This paper is concerned with a family of matroids whose ground sets are the collections of lattice points of regular simplices. To describe such matroids,
Consider for $n,d \in \zz_{\ge 2}$ the $(d-1)$-dimensional simplex
\[
	S_{n,d} := \big\{ (x_1, \dots, x_d) \in \rr^d \mid x_1 + \dots + x_d = n-1 \ \text{ and } \ x_i \ge 0 \ \text{ for } \ 1 \le i \le d \big\},
\]
and let $P_{n,d}$ denote the set of lattice points contained in $S_{n,d}$.
%We call $P_{n,d}$ the \emph{lattice simplex} of the simplex $S_{n,d}$ or simply a \emph{lattice simplex}.
% of \emph{size} $n$ and \emph{dimension} $d-1$.
The set $P_{4,3}$ is illustrated in Figure~\ref{fig:lattice point representation of T4}.
\begin{figure}[t]
	\centering
	\includegraphics[width = 6.5cm]{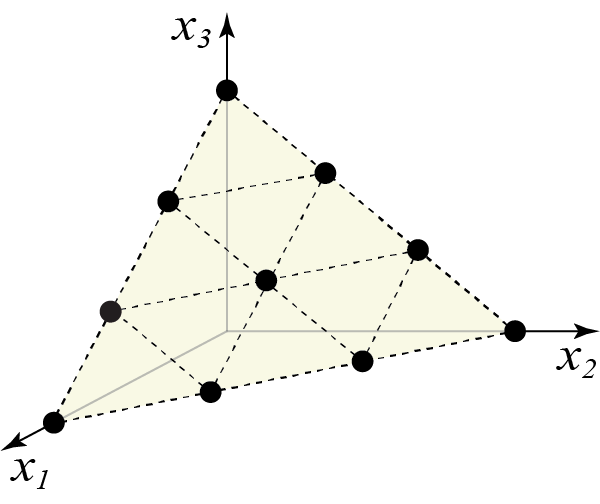}
	\caption{The set of lattice points $P_{4,3}$ and a (shaded) region of the plane in $\rr^3$ containing it. The $10$ lattice points of $P_{4,3}$ are depicted by dark black dots.}
	\label{fig:lattice point representation of T4}
\end{figure}
We denote by $\mathcal{I}_{n,d}$ the collection of all subsets $I$ of $P_{n,d}$ such that, for each $k \le n$, every parallel translate of $P_{k,d}$ contains at most $k$ lattice points of $I$. It has been proved in~\cite{AB07} that $\mathcal{I}_{n,d}$ is the collection of independent sets of a matroid $\T_{n,d}$ with ground set $P_{n,d}$. Here we study some aspects of the combinatorial structure of the matroids $\T_{n,3}$. The case $d = 3$ is particularly important because, as it was proved in \cite{AB07}, the matroid $\T_{n,3}$ is cotransversal and the contransversality property of $\T_{n,3}$ allows to construct the Schubert-generic line arrangement $\mathbf{E}_{n,3}$ explicitly (see \cite[Proposition~9.2]{AB07}). Finally, it is worthy to notice that although contransversal matroids have been the subject of a great deal of investigation, cotransversal matroids other than $\T_{n,3}$ do not seem to have been studied before in connection with tilings.

The original motivation to study the matroids $\T_{n,d}$ comes from \cite{AB07}, where the authors were interested in understanding the set $\mathbf{E}_{n,d}$ of one-dimensional intersections of complete complex flag arrangements.
%For $1 \le k \le d$, let $E_\bullet^k = \big\{ \{0\} = E_0^k \subset E_1^k \subset \dots \subset E_n^k = \cc^n \big\}$ be a generic complete flag, i.e., each $E_i^k$ is an $i$-dimensional subspace of $\cc^n$. In addition, let $\mathbf{E}_{n,d}$ be the set of all lines of the form $E^1_{j_1} \cap E^2_{j_2} \cap \dots \cap E^d_{j_d}$.
It turns out that the dependence relations among the lines in $\mathbf{E}_{n,d}$ are encoded in $\T_{n,d}$. As a result, the structure of $\T_{n,d}$ was crucial to understand the linear dependence of line arrangements resulting from intersecting complete flags of $\cc^n$ and, as a byproduct, to facilitate certain computations on the cohomology ring of the flag manifold. The matroids $\T_{3,d}$ have been studied in~\cite{AC13} in connection with acyclic permutations and fine mixed subdivisions of simplices~\cite{fS05}.

Questions about tilings come in many diverse flavors, from problems about existence and enumeration to problems about computational complexity and feasibility, and they have been investigated in connection with many fields, including combinatorial group theory~\cite{CL90}, algebraic geometry~\cite{fB82a}, computational complexity theory \cite{BNRR95}, and stochastic processes~\cite{LRS01}. Furthermore, tiling theory also finds applications to perfect matching~\cite{tC96}, classical geometric problems~\cite{rK96}, genetic~\cite{BPP12}, \emph{etc}. See~\cite{AS10} for a friendly survey on tileability. Here we establish various cryptomorphic characterizations of the matroids $\T_{n,3}$ in terms of tilings of $S_{n,3}$ into unit triangles, rhombi, and trapezoids.

The remainder of this paper is structured as follows. In Section~\ref{sec:preliminary} we provide some background related to matroids and a few notations needed in later sections. Section~\ref{sec:independent sets} is devoted to characterize the independent sets of $\T_{n,3}$. In Section~\ref{sec:rank and tilings}, we study the rank function of $\T_{n,3}$, showing how certain tilings of $S_{n,3}$ into unit rhombi and unit triangles can encode the size and rank of a given subset of the ground set of $\T_{n,3}$. In Section~\ref{sec:rank and circuits}, we provide a characterization of the circuits of $\T_{n,3}$ by tilings $S_{n,3}$ into unit triangles, rhombi and trapezoids. Finally, in Section~\ref{sec:flats}, we give a geometric characterization of the flats of $\T_{n,3}$.
%Finally, in Section~\ref{sec:connectedness}, we use tilings to show that the matroids $\T_{n,3}$ are connected.
%(Theorem~\ref{thm:tiling matroids are connected}).
%Although we prove connectedness in any dimension greater than one, we exhibit a tiling-based proof for the $2$-dimensional case (Proposition~\ref{prop:tiling matroids are connected}). This proof is not only more elegant, but also illustrates how tilings are intrinsically related to the matroidal structure of $\T_{n,3}$. \\

\section{Preliminary} \label{sec:preliminary}

%In this section we will provide some basic definitions related to matroids we shall be using later. In addition, we will present a useful characterization of the bases of $\T_{n,3}$ that was established in \cite{AB07} and will be used and generalized here. For the general theory of matroids we mostly follow the notation in Oxley \cite{jO92}. 
There are many equivalent axiom systems we can use to define a matroid. Following~\cite{jO92}, we define a matroid via independent sets and take the matroid descriptions via bases and circuits as characterizations.

\begin{definition}
	Let $E$ be a finite set, and let $\ii$ be a collection of subsets of $E$. The ordered pair $(E, \ii)$ is called a matroid if the following properties hold: %\vspace{2pt}
	\begin{enumerate}
		\item the collection $\ii$ contains the empty set; \vspace{4pt}
		\item if $I_1 \in \ii$ and $I_2 \subseteq I_1$, then $I_2 \in \ii$; \vspace{3pt}
		\item if $I_1, I_2 \in \ii$ and $|I_1| < |I_2|$, then there exists $x \in I_2 \setminus I_1$ such that $I_1 \cup \{x\} \in \ii$.
	\end{enumerate}
\end{definition}

Property~(3) is known as the \emph{independence augmentation property}. Let $M = (E, \ii)$ be a matroid. The set $E$ is called the \emph{ground set} of $M$. Also, the elements of $\ii$ are called \emph{independent sets} of $M$, while the subsets of $E$ which are not in $\ii$ are called \emph{dependent sets} of $M$. A maximal independent set is called a \emph{basis} of $M$. A matroid can be also characterized by its set of bases.

\begin{theorem} \cite[Theorem~1.2.3]{jO92}
	Let $E$ be a finite set, and let $\bb$ be a collection of subsets of $E$ such that the following conditions hold: \vspace{4pt}
	\begin{enumerate}
		\item the collection $\bb$ is nonempty; \vspace{3pt}
		\item for all $B_1, B_2 \in \bb$ and $x \in B_1 \! \setminus \! B_2$, there exists $y \in B_2 \! \setminus \! B_1$ satisfying that $(B_1 \! \setminus \! \{x\}) \cup \{y\} \in \bb$.
	\end{enumerate}
	If $\ii := \{ I \subseteq E \mid I \subseteq B \ \text{for some} \ B \in \mathcal{B} \}$, then $(E, \ii)$ is a matroid whose collection of bases is precisely $\bb$.
\end{theorem}
Any two bases of a matroid $M$ have the same cardinality, which is called \emph{rank} of $M$ and is denoted by $r(M)$. In addition, if $S$ is a subset of $E$, then $(S, \mathcal{I}|S)$, where $\mathcal{I}|S := \{I \subseteq S \mid I \in \mathcal{I}\}$, is a matroid called the \emph{restriction} of $M$ to $S$. The \emph{rank} of $S$, denoted by $r(S)$, is the cardinality of any basis of $(S, \mathcal{I}|S)$.

%As mentioned in the introduction, here we are interested in matroids whose ground sets are $T_{n,d}$, the set of lattice points contained in the $(d-1)$-dimensional simplex
%\[
%	S_{n,d} := \big\{ (x_1, \dots, x_d) \in \rr^d \mid x_1 + \dots + x_d = n-1 \ \text{ and } \ x_i \ge 0 \ \text{ for } \ 1 \le i \le d \big\},
%\]
%and whose collections of independents, $\mathcal{I}_{n,d}$, consist of all $I \subset T_{n,d}$ such that $|I \cap T| \le k$ for every $k \le n$ and every parallel translate $T$ of $T_{k,d}$. \vspace{2pt}

As mentioned in the introduction, it was poved in~\cite{AB07} that for each $n,d \in \nn$, the pair $\T_{n,d} = (P_{n,d}, \mathcal{I}_{n,d})$ is a matroid.

%\begin{theorem} \cite[Theorem~4.1]{AB07} \label{thm:tiling matroid}
%	For each $n,d \in \nn$, the pair $(T_{n,d}, \ii_{n,d})$ is a matroid.
%\end{theorem}

\begin{definition}
	For $n \in \nn$, we call $\T_{n,3}$ a \emph{tiling matroid} and denote it simply by $\T_n$.
\end{definition}

It is not hard to verify that the matroid $\T_n$ has rank $n$ and its ground set has size $n(n+1)/2$. From now on we let $T_n$ denote the convex hull of $P_{n,3}$ and think of elements in $P_{n,3}$ as triangles in a regular subdivision of $T_n$ as follows. We tacitly assume that $T_n$ is placed as in the top-right picture of Figure~\ref{fig:the 2-simplex of size 4}, namely that $T_n$ is in the plane of the paper and has a horizontal base. In addition, we think of a lattice point $p$ of $P_{n,3}$ as a closed equilateral triangle pointing upward, centered at $p$, whose side length is the minimal distance between lattice points in $T_n$. Finally, we rescale $T_n$ so that the ground set of $\T_n$ consists of unit triangles. This transition from lattice points to unit triangles is illustrated in~Figure~\ref{fig:the 2-simplex of size 4}.

\begin{figure}[h]
	\centering
	\includegraphics[width = 15cm]{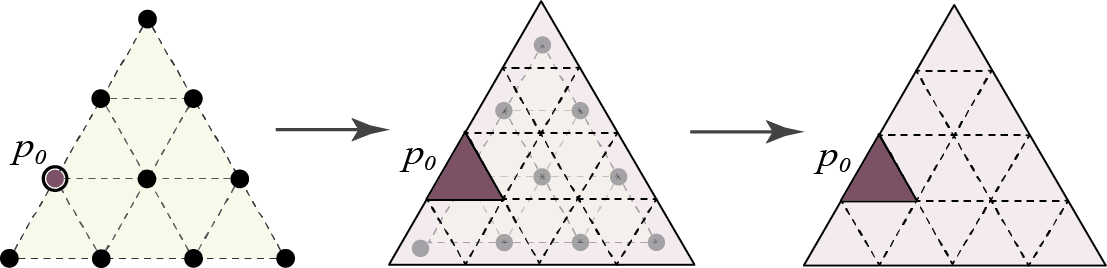}
	\caption{The leftmost picture shows $P_{4,3}$ and an element $p_0 \in P_{4,3}$. The next two pictures show the triangular representation of $p_0$. The rightmost picture in Figure~\ref{fig:the 2-simplex of size 4} illustrates $T_4$ and the unit triangle $p_0$.}
	\label{fig:the 2-simplex of size 4}
\end{figure}

\begin{definition}
	We call the unit triangles representing lattice points of $P_{n,3}$ \emph{unit upward triangles} of $T_n$ and we call the unit triangles inside $T_n \setminus \cup_{\Delta \in P_{n,3}} \Delta$ \emph{unit downward triangles} of $T_n$. Let $\mfu(T_n)$  (resp., $\mfd(T_n)$) denote the set of unit upward (resp., downward) triangles of $T_n$.
\end{definition}

Then $\mfu(T_n)$ is the ground set of $\T_n$. We say that a nonempty subset of $T_n$ is a \emph{lattice region} if its closure is the union of unit triangles of $T_n$. Note that any lattice upward triangle $T$ of $T_n$ is a parallel translate of $T_\ell$ for some $\ell \le n$; in this case we call $\ell$ the \emph{size} of $T$ and set $\text{size}(T) := \ell$. If $A \subseteq T_n$ is a lattice region, then we define
\[
	\mfu(A) := \{X \in \mfu(T_n) \mid X \subseteq A\} \quad \text{ and } \quad \mfd(A) := \{X \in \mfd(T_n) \mid X \subseteq A\}.
\]
On the other hand, given a collection $\mathfrak{s}$ of lattice regions of $T_n$, we set
\[
	A(\mathfrak{s}) := \bigcup_{R \in \mathfrak{s}} R.
\]
The \emph{triangular hull} of $\mathfrak{s}$ is the smallest lattice upward triangle of $T_n$ containing all lattice regions in $\mathfrak{s}$. The concepts in the following two definitions are central in our exposition.

\begin{definition}
	For $\mfs \subseteq \mfu(T_n)$, we call $T_n \setminus A(\mfs)$ the \emph{holey region} corresponding to $\mfs$. 
\end{definition}

\begin{definition}
	For a lattice region $R$ of $T_n$, we call $\mft$ a \emph{tiling} of $R$ provided that $\mft$ consists of closed lattice regions of $T_n$ whose interiors are pairwise disjoint and $\cup_{T \in \mft} T$ equals the closure of $R$.
\end{definition}

Additionally, let us introduce notation for one of the most important lattice regions and tilings we will consider in this paper.

\begin{definition}
	Given a tiling matroid $\T_n$, we call the union of two adjacent unit triangles of $T_n$ a \emph{unit rhombus}. A tiling into unit rhombi of a lattice region $R$ of $T_n$ is called a \emph{lozenge tiling} of $R$.
\end{definition}
%\medskip

Figure~\ref{fig:the three rhombi} illustrates all possible unit rhombi of $T_n$ (up to translation): one vertical and two (symmetric) horizontal. We say that a unit rhombus $R$ of $T_n$ is \emph{horizontal} provided that one of its sides is horizontal; otherwise, we say that $R$ is \emph{vertical}. \\

\begin{figure}[h]
	\centering
	\raisebox{0.5 \height}{\includegraphics[width = 2cm]{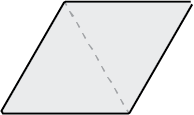}}
	\hspace{1cm}
	\includegraphics[width = 1.4cm]{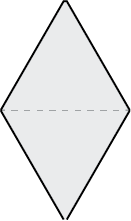}
	\hspace{1cm}
	\raisebox{0.5 \height}{\includegraphics[width = 2cm]{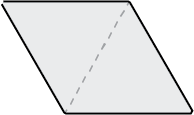}}
	\caption{The three unit rhombi of $T_n$ up to translation.}
	\label{fig:the three rhombi}
\end{figure}
\medskip

A lozenge tiling of the holey region corresponding to a $4$-element subset of $\mfu(T_4)$ is illustrated in the left picture of Figure~\ref{fig:a basis and a non-basis}. The following characterization of the bases of $\T_n$ was established in \cite{AB07}.

\begin{theorem} \cite[Theorem~6.2]{AB07} \label{thm:characterization of bases}
	For the tiling matroid $\T_n$, let $\mathfrak{b}$ be a subset of $\mfu(T_n)$. Then $\mathfrak{b}$ is a basis of $\T_n$ if and only if there exists a lozenge tiling of $T_n \setminus A(\mathfrak{b})$.
\end{theorem}

The following example illustrates Theorem~\ref{thm:characterization of bases}.

\begin{example}
	Consider the tiling matroid $\T_4$. Let $\mfb$ consist of the dark unit upward triangles in the left picture of Figure~\ref{fig:a basis and a non-basis}. Note that $\mfb$ is a basis of $\T_4$. A lozenge tiling of the holey region $T_4 \setminus A(\mfs)$ is shown. On the other hand, let $\mfs$ be the set of dark unit upward triangles in the right picture of Figure~\ref{fig:a basis and a non-basis}. One can easily see that the holey region $T_4 \setminus A(\mfs)$ cannot be tiled into unit rhombi.
	\begin{figure}[h]
		\begin{center}
			\includegraphics[width=4cm]{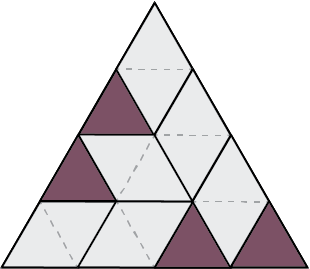}
			\hspace{2cm}
			\includegraphics[width=4cm]{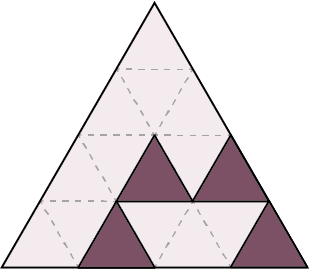}
			\caption{On the left, a basis of $\T_4$. On the right, a size-$4$ subset of $\mfu(T_4)$ that is not a basis of $\T_4$.}
			\label{fig:a basis and a non-basis}
		\end{center}
	\end{figure}
\end{example}

%\begin{remark} \label{rmk:contransversality}
%	It was also proved in \cite[Section~6]{AB07} that the tiling matroid $\T_n$ is cotransversal. For the definition and properties of cotransversal matroids, see \cite[Section~2.4]{jO92}.
%\end{remark}

\section{Tiling Characterization of the Independent Sets} \label{sec:independent sets}

In this section, we characterize the independent sets $\mfs$ of the matroid $\T_n$ in terms of certain tilings of $T_n \setminus A(\mfs)$. This characterization generalizes that one of bases given in Theorem~\ref{thm:characterization of bases}.

\begin{definition} \label{def:type-1 trapezoids}
	A \emph{type-1 trapezoid} of $T_n$ is a lattice trapezoid of $T_n$ that is the union of two unit upward triangles and one unit downward triangle. \\
\end{definition}

As in the case of unit rhombi, we say that a type-1 trapezoid $T$ of $T_n$ is \emph{horizontal} if it has its two parallel sides horizontal. Up to translation, there are three type-1 trapezoids, as depicted in Figure~\ref{fig:the three type-1 trapezoids}.

\begin{figure}[h]
	\centering
	\includegraphics[width = 1.8cm]{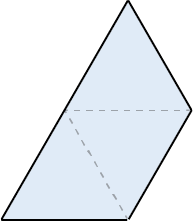}
	\hspace{1cm}
	\raisebox{0.3 \height}{\includegraphics[width = 2.6cm]{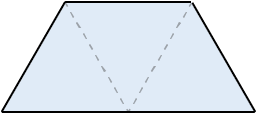}}
	\hspace{1cm}
	\includegraphics[width = 1.8cm]{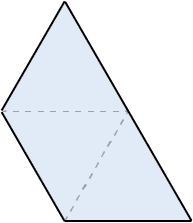}
	\caption{The three type-1 trapezoids up to translation.}
	\label{fig:the three type-1 trapezoids}
\end{figure}

\begin{theorem} \label{thm:characterization of independent sets}
	Let $\mfs$ be a subset of $\mfu(T_n)$. Then $\mfs$ is independent if and only if the lattice region $T_n \setminus A(\mfs)$ can be tiled using unit rhombi and exactly $n - |\mfs|$ type-1 trapezoids.
\end{theorem}

\begin{proof}
	To prove the direct implication, suppose that $\mfs$ is an independent set of $\T_n$. If $|\mfs| = n$, then $\mfs$ is a basis of $\T_n$, and we are done by Theorem~\ref{thm:characterization of bases}. So we assume that $|\mfs| < n$. Take a basis $\mfb$ of $\T_n$ containing $\mfs$. Theorem~\ref{thm:characterization of bases} ensures the existence of a lozenge tiling of $T_n \setminus A(\mfb)$. Let $\mft$ be one of such tilings. Merging some of the unit upward triangles in $\mfb \setminus \mfs$ with some of their adjacent rhombi, we create new tilings $\mft'$ of $T_n \setminus A(\mfs)$ consisting of $m$ type-1 trapezoids ($m \le n - |\mfs|$), some unit rhombi, and some unit upward triangles. Among all such tilings, let $\mft'$ be one maximizing $m$ and, suppose, by way of contradiction, that $m < n - |\mfs|$. Then there is a unit upward triangle $X \in \mfb \setminus \mfs$ that did not merge to any unit rhombus of $\mft$. By rotating $T_n$ if necessary, we can assume that $X$ is adjacent to a vertical unit rhombus in $\mft$. Now we consider two cases.
	
	CASE 1: The horizontal side $a$ of $X$ is not in the horizontal border of $T_n$. The fact that $X$ is adjacent to a vertical unit rhombus in $\mft$ forces $n \ge 3$. Let $R$ denote one of the vertical unit rhombus adjacent to $X$ in $\mft$. Because $X$ did not merge in $\mft'$, another unit upward triangle $Y \in \mfb \setminus \mfs$ merged to $R$ in $\mft'$ creating a type-1 trapezoid. Assume, without loss of generality, that $Y$ is right after $X$ in the same row. In this case, $a$ is a side of the horizontal unit rhombus $R'$ right below $X$ in $\mft$. As $X$ did not merge in $\mft'$, another unit upward triangle $Z \in \mfb \setminus \mfs$ merged to $R'$ in $\mft'$ creating a type-1 trapezoid. Note that $X$, $Y$, and $Z$ belong to the same size-$3$ lattice upward triangle of $T_n$, which can be re-tiled by using exactly three type-1 trapezoids in such a way that no two triangles in $\{X,Y,Z\}$ are part of the same trapezoid. But this yields a tiling of $T_n \setminus A(\mfs)$ containing more type-$1$ trapezoids than $\mft'$ does, which contradicts the maximality of $m$.
		
	CASE 2: The horizontal side $a$ of $X$ is in the horizontal border of $T_n$. As the cases when $n < 3$ follow straightforwardly, we assume $n \ge 3$. Let $\mathfrak{x} = \{X_1, \dots, X_t\}$ be the maximal set of consecutive unit upward triangles in the bottom row of $T_n$ such that $X \in \mathfrak{x} \subseteq \mfb \setminus \mfs$. Notice that if $t=1$, then the maximality of $\mathfrak{x}$ would make one of the vertical unit rhombi adjacent to $X$ in $\mft$ available to merge with $X$ to form a tiling of $T_n \setminus A(\mfs)$ containing more type-1 trapezoids than $\mft'$ does, which is a contradiction. Hence we also assume that $t > 1$.
	
	CASE 2.1: $X \in \{X_1, X_t\}$. Assume, without loss of generality, that $X = X_1$. We first suppose that $X$ is not a corner of $T_n$. Then the unit rhombus $R$ adjacent from the left to $X$ in $\mft$ must be horizontal by the maximality of $\mathfrak{x}$. In addition, as $X$ did not merge in $\mft'$, the unit rhombus $R$ must have merged in $\mft'$ with the unit upward triangle $Y$ right on top of it. Also, note the rhombus $R'$ adjacent from the right to $X$ in $\mft$ is vertical and must have merged to $X_2$ in $\mft'$ (by the maximality of $m$). Now the size-$3$ lattice upward triangle containing $X$, $X_2$, and $Y$ can be re-tiled into three type-1 trapezoids such that not two triangles in $\{X,X_2,Y\}$ are part of the same trapezoid. However, this results in a new tiling of $T_n \setminus A(\mfs)$ having more type-$1$ trapezoids than $\mft'$ does, which is a contradiction.
	
	On the other hand, suppose that $X = X_1$ is a corner of $T_n$. It follows from $|\mfs| \ge 1$ that $t < n$. As a result, $X_t$ cannot be a corner of $T_n$. Notice that for every $i \in [t-1]$ the common adjacent tile to $X_i$ and $X_{i+1}$ in $\mft$ is a vertical unit rhombus. If for some $i \in [t] \setminus \{1\}$ the triangle $X_i$ did not merge with its left adjacent rhombus in $\mft'$, then we can re-tile $T_n \setminus A(\mfs)$ by merging $X_j$ to its right adjacent vertical unit rhombus for all $j \in [i-1]$ to create a tiling of $T_n \setminus A(\mfs)$ containing more type-$1$ trapezoids than $\mft'$ does, which is not possible. Then we can assume that $X_t$ merges in $\mft'$ to its left vertical unit rhombus. By the maximality of $\mathfrak{x}$, the unit rhombus $R''$ adjacent from the right to $X_t$ in $\mft$ must be horizontal. If $R''$ does not merge in $\mft'$, then we can re-tile $T \setminus A(\mfs)$ by merging each of the $X_i$ to its right adjacent unit rhombus, obtaining once again a tiling of $T_n \setminus A(\mfs)$ with more than $m$ type-$1$ trapezoids. Then suppose that $R''$ merges in $\mft'$ to the unit upward triangle $Z$ right on top of it. In this case, we can merge $X_t$ with $R''$, $Z$ with its top-left vertical unit rhombus, and $X_i$ with its top-right unit rhombus for each $i \in [t-1]$ to obtain a tiling of $T_n \setminus A(\mfs)$ containing more type-1 trapezoids than $\mft'$ does. However, this contradicts the maximality of $m$.
	
	CASE 2.2: $X \notin \{X_1, X_t\}$. Let $X = X_j$. Because $t < n$, either $X_1$ or $X_t$ is not a corner of $T_n$. Suppose, without loss of generality, that $X_t$ is not a corner of $T_n$. In this case, we can proceed as we did in the second paragraph of CASE~2.1 to obtain a new tiling of $T_n \setminus A(\mfs)$ having more than $m$ type-1 trapezoids, which once again contradicts the maximality of $m$.
	
	To check the reverse implication, suppose that $\mft$ is a tiling of the region $T_n \setminus A(\mfs)$ consisting of unit rhombi and $n - |\mfs|$ type-1 trapezoids. Now split each type-1 trapezoid of $\mft$ into a unit rhombus and a unit upward triangle, and let $\mfs'$ denote the set of all unit upward triangle resulting from such splittings. Then we have a tiling $\mft'$ of $T_n \setminus A(\mfs \cup \mfs')$ using only unit rhombi. By Theorem~\ref{thm:characterization of bases}, the set $\mfs \cup \mfs'$ is a basis of $\T_n$. Hence $\mfs$ is an independent set of $\T_n$, which completes the proof.
\end{proof}

The following example illustrates the characterization established in Theorem~\ref{thm:characterization of independent sets}.

\begin{example}
	Consider the tiling matroid $\T_4$. The left picture of Figure~\ref{fig:independent and non-independent examples} shows an independent set $\mfs$ of $\T_4$ whose elements are depicted by the three dark unit upward triangles. It also shows a tiling of the holey region corresponding to $\mfs$ into unit rhombi and a type-$1$ trapezoid. By contrast, the right picture of Figure~\ref{fig:independent and non-independent examples} illustrates a dependent set $\mfs'$ of $\T_4$. Notice that the holey region corresponding to $\mfs'$ cannot be tiled into unit rhombi and type-$1$ trapezoids as it consists of more unit downward triangles than unit upward triangles.
	\begin{figure}[h]
		\begin{center}
			\includegraphics[width=4cm]{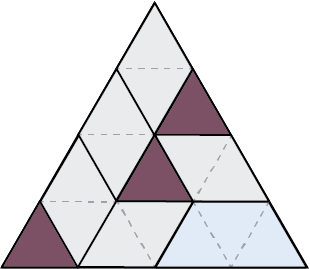}
			\hspace{2cm}
			\includegraphics[width=4cm]{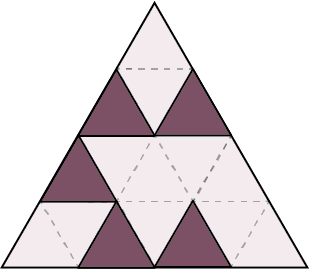}
			\caption{On the left, an independent set of $\T_4$. On the right, a dependent set of $\T_4$.}
			\label{fig:independent and non-independent examples}
		\end{center}
	\end{figure}
\end{example}
%\medskip

\section{Rank and Tilings} \label{sec:rank and tilings}

Given a subset $\mfs$ of $\mfu(T_n)$, we now study how the size and rank of $\mfs$ can be used to count the number of congruent pieces of any tiling of $T_n \setminus A(\mfs)$ into unit triangles and unit rhombi that maximizes the number of rhombi.

Let $R$ be a lattice region of $T_n$, and let $\mft$ be a tiling of $R$ consisting of unit triangles and unit rhombi. We call a tiling $\mft'$ of $R$ a \emph{reconfiguration} of $\mft$ provided that $\mft'$ contains the same numbers of unit upward triangles, unit downward triangles, and unit rhombi that $\mft$ does. Before presenting the main result of this section, let us collect the following lemma.

\begin{lemma} \label{lem:pushing up a downward triangle}
	Let $\mfs$ be a subset of $\mfu(T_n)$, and let $\mft$ be a tiling of $T_n \setminus A(\mfs)$ into unit rhombi and unit triangles. Then there exists a reconfiguration of $\mft$ all whose unit downward triangles are adjacent from below to unit upward triangles.
\end{lemma}

\begin{proof}
	Among all reconfigurations of $\mft$, let $\mft'$ be one maximizing the number $N$ of unit downward triangles that are adjacent from below to unit upward triangles. Let $\mfd$ and $\mfu$ be the sets of unit downward triangles and unit upward triangles in $\mft'$, respectively. Suppose, by way of contradiction, that there exists $X \in \mfd$ that is not adjacent from below to any unit upward triangle. Then $X$ must be adjacent from below to a horizontal unit rhombus of $\mft'$ and, therefore, we can move $X$ one unit row up in $T_n$ by using the moves illustrated in Figure~\ref{fig:shifting downward triangle and rhombus}, obtaining a reconfiguration $\mft''$ of $\mft'$.
	\vspace{5pt}
	\begin{figure}[h]
		\centering
		\includegraphics[width = 1.8cm]{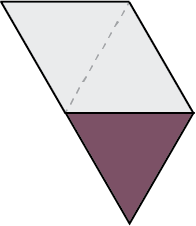}
		\hspace{0.1cm}
		\raisebox{3.4 \height}{\includegraphics[width = 1.2cm]{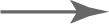}}
		\hspace{-0.1cm}
		\includegraphics[width = 1.8cm]{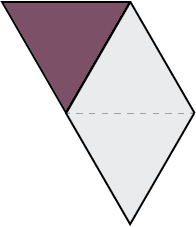}
		\hspace{2.8cm}
		\includegraphics[width = 1.8cm]{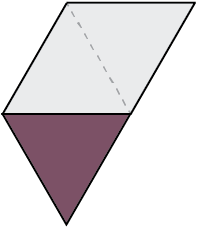}
		\hspace{-0.1cm}
		\raisebox{3.4 \height}{\includegraphics[width = 1.2cm]{MoveArrow}}
		\hspace{0.1cm}
		\includegraphics[width = 1.8cm]{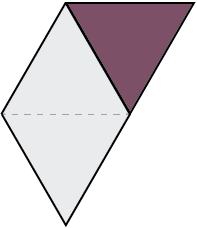}
		\caption{}
		\label{fig:shifting downward triangle and rhombus}
	\end{figure}
	%\vspace{-5pt}
	Notice that the triangles in $\mfd \setminus \{X\}$ that were adjacent from below in $\mft'$ to a unit rhombus (or a unit upward triangle) keep this property in $\mft''$. As $N$ is maximal, $X$ is still adjacent from below in $\mft''$ to a unit rhombus. Then we can actually move between reconfigurations of $\mft'$ by performing the move in Figure~\ref{fig:shifting downward triangle and rhombus} to $X$ until it reaches the second unit row of $T_n$ (from top to bottom), obtaining a final reconfiguration of $\mft'$ where $X$ is adjacent from below to the top unit triangle of $T_n$. However, this contradicts the maximality of $N$. Hence each unit downward triangle in $\mft'$ must be adjacent from below to a unit upward triangle, and the proof follows.
\end{proof}

The next result establishes a relation between the size and rank of subsets $\mfs$ of $\mfu(T_n)$ and certain tilings of $T_n \setminus A(\mfs)$ into unit triangles and unit rhombi.

\begin{prop} \label{prop:rank characterization}
	If $\mfs \subseteq \mfu(T_n)$, then a tiling of $T_n \setminus A(\mfs)$ into unit rhombi and unit triangles with maximum number of rhombi must contain $|\mfs| - r(\mfs)$ unit downward triangles.
\end{prop}

\begin{proof}
	Let us first argue that there exists a tiling of $T_n \setminus A(\mfs)$ into unit rhombi and unit triangles having exactly $|\mfs| - r(\mfs)$ unit downward triangles. Let $\mfr$ be an independent set of $\T_n$ contained in $\mfs$ such that $|\mfr| = r(\mfs)$. Now take a basis $\mfb$ of $\T_n$ containing $\mathfrak{r}$. As any subset of $\mfb$ is an independent set of $\T_n$, the sets $\mfb \setminus \mathfrak{r}$ and $\mfs \setminus \mathfrak{r}$ are disjoint. By Theorem~\ref{thm:characterization of bases}, there exists a lozenge tiling $\mft$ of $T_n \setminus A(\mfb)$. This, in turns, gives us a tiling $\mft'$ of $T_n \setminus A(\mathfrak{r})$ consisting of unit rhombi and the unit upward triangles in $\mfb \setminus \mathfrak{r}$. Since $\mfb \setminus \mathfrak{r}$ and $\mfs \setminus \mathfrak{r}$ are disjoint, each $X \in \mfs \setminus \mathfrak{r}$ must be covered by a rhombus $R_X$ of $\mft'$. After splitting each rhombus $R_X$ into two unit triangles, one obtains the desired tiling of $T_n \setminus A(\mathfrak{s})$ into unit rhombi and unit triangles containing exactly $|\mfs| - r(\mfs)$ unit downward triangles.
	
	Now observe that the number of rhombi in a tiling $\mft$ of $T_n \setminus A(\mfs)$ into unit rhombi and unit triangles determines the number of unit upward triangles and the number of unit downward triangles in $\mft$. Indeed, it is easy to verify that if $\mft$ contains $m$ unit rhombi, then it must contain $\binom{n}{2} - m$ unit downward triangles and $\binom{n+1}{2} - m - |\mfs|$ unit upward triangles. In particular, $\mft$ maximizes the number of unit rhombi if and only if it minimizes the number of unit downward triangles. Hence we are done once we prove that every tiling of $T_n \setminus A(\mfs)$ into unit rhombi and unit triangles contains at least $|\mfs| - r(\mfs)$ unit downward triangles.
	
	Take a tiling $\mft$ of $T_n \setminus A(\mfs)$ into unit rhombi and unit triangles minimizing the number of unit downward triangles. Let $\mfd$ and $\mfu$ be the sets of unit downward triangles and unit upward triangles of $\mft$, respectively. By Lemma~\ref{lem:pushing up a downward triangle}, there is no loss in assuming that all triangles in $\mfd$ are adjacent from below to unit upward triangles. On the other hand, the minimality of $|\mfd|$ ensures that no triangle in $\mfd$ is adjacent from below to a triangle in $\mfu$. Hence each triangle of $\mfd$ is adjacent from below to a triangle of $\mfs$. Merging each triangle in $\mfd$ with its corresponding adjacent from above triangle of $\mfs$, one obtains a lozenge tiling $\mft'$ of a lattice region $R$ satisfying that $\mfu(R) \cap \mfs$ consists precisely of those triangles in $\mfs$ that did not merge with the $|\mfd|$ unit downward triangles of $\mft$. By Theorem~\ref{thm:characterization of bases}, the subset $\mfu(R) \cap \mfs$ of $\mfs$ is independent. This implies that $|\mfs| - |\mfd| = |\mfu(R) \cap \mfs| \le r(\mfs)$ and, therefore, $|\mfd| \ge |\mfs| - r(\mfs)$, which completes the proof.
\end{proof}

The following corollary is an immediate consequence of Proposition~\ref{prop:rank characterization}.

\begin{cor} \label{cor:numerology of the tilings}
	If $\mfs \subseteq \mfu(T_n)$ and $\mft$ is a tiling of $T_n \setminus A(\mfs)$ into unit triangles and unit rhombi that maximizes the number of unit rhombi, then $\mft$ contains $\binom{n}{2} - (|\mfs| - r(\mfs))$ unit rhombi, $|\mfs| - r(\mfs)$ unit downward triangles, and $n - r(\mfs)$ unit upward triangles.
\end{cor}

\begin{remark}
	Notice that Corollary~\ref{cor:numerology of the tilings} does not imply Theorem~\ref{thm:characterization of independent sets} as some of the $n - r(\mfs)$ unit upward triangles in the tiling $\mft$ might be seated at the bottom line of the lattice region $T_n$ and, therefore, do not have rhombi right below them to merge.
\end{remark}

\section{Tiling Characterization of the Circuits} \label{sec:rank and circuits}

We proceed to characterize the circuits $\mfc$ of $\T_n$ in terms of certain tilings of $T_n \setminus A(\mfc)$ into unit rhombi and (possibly reflected) type-$1$ trapezoids.

Let $M$ be a matroid. A subset $C$ of the ground set of $M$ is called a \emph{circuit} of $M$ if $|C| - r(C) = 1$ and $C \setminus \{x\}$ is an independent set of $M$ for each $x \in C$, that is, $C$ is a minimal dependent set of $M$. On the other hand, a \emph{loop} (resp., \emph{parallel}) of $M$ is a dependent set of $M$ of size one (resp., two).

\begin{remark}
	Tiling matroids are \emph{simple}, i.e., they contain neither loops nor parallels. 
\end{remark} 

\begin{example} \label{ex:circuit and non-circuit examples}
	Consider the tiling matroid $\T_4$. The left picture of Figure~\ref{fig:circuit and non-circuit examples} shows a circuit of rank $2$ (and size $3$) whose elements are depicted by the three dark unit upward triangles. It is easy to argue that all circuits of $\T_n$ have the same shape, meaning that they are geometrically equal up to translation. On the other hand, the right picture of Figure~\ref{fig:circuit and non-circuit examples} shows a subset $\mfs$ of $\mfu(T_4)$ of rank $4$ that is not a circuit even though $|\mfs| - r(\mfs) = 1$. Note that $\mfs$ contains a circuit of rank $2$.
	\begin{figure}[h]
		\begin{center}
			\includegraphics[width=4cm]{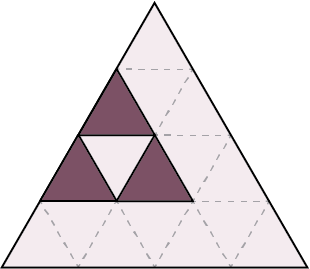}
			\hspace{2cm}
			\includegraphics[width=4cm]{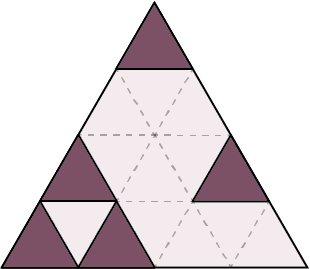}
			\caption{On the left, a circuit of rank $2$. On the right, a non-circuit of rank $4$.}
			\label{fig:circuit and non-circuit examples}
		\end{center}
	\end{figure}
\end{example}
\medskip

\begin{definition}
	A \emph{type-2 trapezoid} is a lattice trapezoid of $T_n$ which is the union of one unit upward triangle and two unit downward triangles (cf. Definition~\ref{def:type-1 trapezoids}). 
\end{definition}
\medskip

As in the case of type-1 trapezoids, one has three possible type-2 trapezoids (up to translation), one of them being horizontal. Observe that they are the reflection of the type-1 trapezoids through their largest sides.
\medskip

\begin{figure}[h]
	\centering
	\includegraphics[width = 1.8cm]{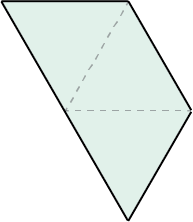}
	\hspace{1cm}
	\raisebox{0.5 \height}{\includegraphics[width = 2.6cm]{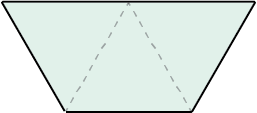}}
	\hspace{1cm}
	\includegraphics[width = 1.8cm]{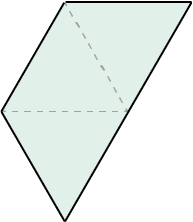}
	\caption{The three type-2 trapezoids up to translation.}
	\label{fig:the three type-2 trapezoids}
\end{figure}

%Now that we have incremented our variety of tiles, it is natural to wonder whether we can extend the characterization given in Theorem~\ref{thm:characterization of independent sets} to non-necessarily independent subsets $\mfs$ of $\mfu(T_n)$ based on the existence of a

To characterize further distinguished subsets $\mfs$ of the ground set of $\T_n$ in terms of tiling of $T_n \setminus A(\mfs)$ into unit rhombi and unit trapezoids, $\mfs$ cannot contain any circuit of rank $2$ because such circuits isolate unit triangles (see Example~\ref{ex:circuit and non-circuit examples}). However, this requirement does not suffice as the following example illustrates.

\begin{example}
	Consider the tiling matroid $\T_4$ along with the subset $\mfs$ of $\mfu(T_4)$ given by the dark unit upward triangles illustrated in Figure~\ref{fig:untilable complement}. As the lattice region $T_4 \setminus A(\mfs)$ consists of only one unit upward triangle and three unit downward triangles, it cannot be tiled into unit rhombi and unit trapezoids.
	\begin{figure}[h]
		\centering
		\includegraphics[width = 4cm]{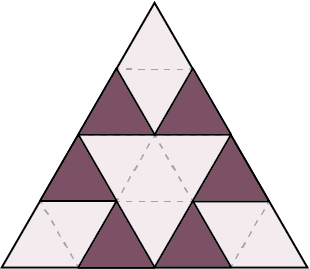}
		\caption{A subset of $\mfu(T_4)$ whose corresponding holey region cannot be tiled into unit rhombi and unit trapezoids.}
		\label{fig:untilable complement}
	\end{figure}
\end{example}

The holey region corresponding to most circuits, however, can be tiled into unit rhombi and unit trapezoids. Indeed, we will see in Theorem~\ref{thm:characterization of circuit} a tiling characterization of circuits of rank greater than $2$. First, we collect a few results we shall be using later. For a subset $\mfs$ of $\mfu(T_n)$, we say that a lattice upward triangle $T$ of $T_n$ of size $k$ is \emph{saturated} by~$\mfs$ if $|\mfs \cap \mfu(T)| = k$, \emph{over-saturated} by $\mfs$ if $|\mfs \cap \mfu(T)| \ge k$, and \emph{strictly over-saturated} by $\mfs$ if $|\mfs \cap \mfu(T)| > k$.

\bigskip
\begin{prop} \label{prop:convex hull of a circuit}
	Let $\mfc$ be a circuit of $\T_n$. Then there exists exactly one lattice upward triangle of $T_n$ that is strictly over-saturated by $\mfc$, namely the triangular hull of $\mfc$.
\end{prop}

\begin{proof}
	Since $\mfc$ is a dependent set of $\T_n$, there must be a lattice upward triangle $T$ of $T_n$ that is strictly over-saturated by $\mfc$. Notice that each $X \in \mfc$ must be inside $T$; otherwise, $T$ would be strictly over-saturated by $\mfc \setminus \{X\}$, contradicting that $\mfc \setminus \{X\}$ is an independent set of $\T_n$. Among all lattice upward triangles strictly over-saturated by $\mfc$, assume that $T$ is minimal under inclusion. The minimality of $T$ implies now that $T$ is the triangular hull of $\mfc$. As $\mfc$ is contained in $\mfu(T)$, it follows that $\text{size}(T) < |\mfc|$. The fact that $\mfc$ contains an independent subset of $\T_n$ of size $|\mfc|-1$ yields $|\mfc| - 1 \le \text{size}(T)$. Hence $\text{size}(T) = |\mfc| - 1$. Finally, let $T'$ be a lattice upward triangle of $T_n$ strictly over-saturated by $\mfc$. Since $T'$ contains $A(\mfc)$, we have that $T \subseteq T'$. This, along with the fact that $T'$ is strictly over-saturated by $\mfc$, guarantees that $\text{size}(T) \le \text{size}(T') \le |\mfc| - 1 = \text{size}(T)$. Thus, $T = T'$, and the uniqueness follows.
\end{proof}

\begin{lemma} \label{lem:filling equilateral trapezoids}
	Each lattice (isosceles) trapezoid of $T_n$ of side length $k$ can be tiled using unit rhombi and $k$ type-1 trapezoids.
\end{lemma}

\begin{proof}
	We can tile each unit row of such a lattice trapezoid by placing a horizontal type-1 trapezoid covering its three leftmost unit triangles and covering the rest of the row with unit rhombi, as illustrated in~Figure~\ref{fig:tiled trapezoid}.
	\begin{figure}[h]
		\begin{center}
			\includegraphics[width=6.5cm]{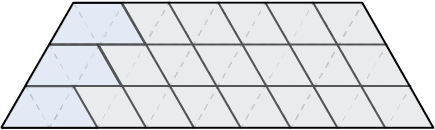}
		\end{center}
		\caption{A lattice trapezoid of side length $3$ tiled using unit rhombi and three type-$1$ trapezoids.}
		\label{fig:tiled trapezoid}
	\end{figure}

\end{proof}

\begin{lemma} \label{lem:tiles in the border of a triangular region}
	Let $T$ be a lattice upward triangle of $T_n$. Then each unit rhombus or type-1 trapezoid of $T_n$ covers at least the same number of unit upward triangles inside $T$ as unit downward triangles inside $T$.
\end{lemma}

\begin{proof}
	Let $S$ be either a unit rhombi or a type-$1$ trapezoid of $T_n$. Set $I = S \cap T$, and let $t$ be the number of unit triangles inside $I$, i.e., $t = |\mfd(I)| + |\mfu(I)|$. Clearly, $t \in \{0,1,2,3\}$. If $t=0$, then $|\mfd(I)| = |\mfu(I)| = 0$. If $t=1$, then $I$ must be a unit upward triangle and so $0 = |\mfd(I)| = |\mfu(I)| - 1$. If $t = 2$, then $I$ must be a unit rhombi and, therefore, $1 = |\mfd(I)| = |\mfu(I)|$. Finally, if $t=3$, then $S$ must be a type-$1$ trapezoid and $I = S$, which implies that $1 = |\mfd(I)| = |\mfu(I)| - 1$. As in any case we have verified that $|\mfd(I)| \le |\mfu(I)|$, the lemma follows.
\end{proof}

We are now in a position to give a characterization of the circuits of $\T_n$.

\begin{theorem} \label{thm:characterization of circuit}
	Let $\mfc$ be a subset of $\mfu(T_n)$ such that $|\mfc| \ge 4$. Then $\mfc$ is a circuit of $\T_n$ if and only if the following two conditions hold: \vspace{2pt}
	\begin{enumerate}
		\item the triangular hull of $\mfc$ is the only lattice upward triangle strictly over-saturated by $\mfc$; \vspace{3pt}
	 	\item the minimum number of type-$2$ trapezoids we can use to tile $T_n \setminus A(\mfc)$ into unit rhombi and unit trapezoids is $1$.
	\end{enumerate}
\end{theorem}

\begin{proof}
	For the forward implication, let $\mfc$ be a circuit. By Proposition~\ref{prop:convex hull of a circuit}, the triangular hull $T$ of $\mfc$ is the unique lattice upward triangle strictly over-saturated by $\mfc$. We proceed to show that condition~(2) also holds. As $\mfc$ is a circuit, $\text{size}(T) = r(\mfc)$. Now fix $X \in \mfc$. Note that $\mfc \setminus \{X\}$ is a basis of the tiling matroid $\T_{r(\mfc)}$ with $T_{r(\mfc)} = T$. Thus, one can use Theorem~\ref{thm:characterization of bases} to obtain a tiling $\mft$ of $T \setminus A(\mfc \setminus \{X\})$ consisting of unit rhombi. In such a tiling, $X$ must be covered by a rhombus $R_X$. Notice that the unit downward triangle of $R_X$ must be adjacent to a unit upward triangle $Y$ not contained in $\mfc$; otherwise, the triangular hull of $R_X$ (which has size $2$) would be strictly over-saturated by $\mfc$, and Proposition~\ref{prop:convex hull of a circuit} would force $|\mfc| = 3$. Let $R_Y$ be the unit rhombus of $\mft$ covering $Y$. Now we can obtain a tiling $\mft'$ of $T \setminus A(\mfc)$ as follows. First keep the tiling configuration of $\mft$ outside $R_X \cup R_Y$, then make $X$ hollow, and finally merge the unit downward triangle of $R_X$ to $R_Y$. Clearly, $\mft'$ consists of unit rhombi and one type-2 trapezoid, namely the new tile containing $Y$.
	
	From the tiling $\mft'$ of $T \setminus A(\mfc)$, we can construct the desired tiling of $T_n \setminus A(\mfc)$ provided we tile $T_n \setminus T$ into unit rhombi and type-$1$ trapezoids. To do this, first split the lattice region $T_n \setminus T$ into three lattice trapezoids as illustrated in Figure~\ref{fig:tiling trapezoids}.
	\begin{figure}[h]
		\begin{center}
			\includegraphics[width=4cm]{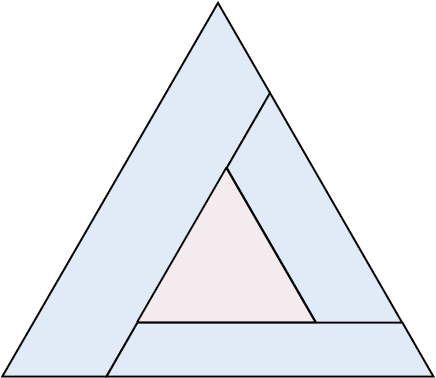}
			\caption{The lattice region $T_n \setminus T$ split into three lattice trapezoids.}
			\label{fig:tiling trapezoids}
		\end{center}
	\end{figure}
	Clearly, the sum of the side lengths of these three lattice trapezoids is $n - r(\mfc)$. Therefore one can use Lemma~\ref{lem:filling equilateral trapezoids} to tile $T_n \setminus T$ into unit rhombi and $n - r(\mfc)$ type-1 trapezoids, obtaining thereby the desired tiling of $T \setminus A(\mfc)$. Finally, notice that $1$ is the minimal number of type-2 trapezoids we can use to tile $T_n \setminus A(\mfc)$ since using $0$ type-2 trapezoids would imply, by Theorem~\ref{thm:characterization of independent sets}, that $\mfc$ is an independent set of $\T_n$.
	
	To prove the backward implication, we first show that $|\mfu(T) \cap \mfc| \le \ell + 1$ for each lattice upward triangle $T$ of $T_n$ of size $\ell$. Suppose, by way of contradiction, that $|\mfu(T) \cap \mfc| \ge k+2$ for some lattice upward triangle $T$ of size $k$. As $|\mfu(T)| = |\mfd(T)| + k$ and $|\mfd(T)| = |\mfd(T \setminus A(\mfc))|$, one finds that
	\[
		|\mfu(T \setminus A(\mfc))| = |\mfu(T)| - |\mfu(T) \cap \mfc| \le |\mfu(T)| - (k+2) = |\mfd(T \setminus A(\mfc))| - 2.
	\]
	Then $T \setminus A(\mfc)$ contains at least two more unit downward triangles than unit upward triangles. By Lemma~\ref{lem:tiles in the border of a triangular region}, every unit rhombus or type-1 trapezoid in any tiling of $T_n \setminus A(\mfc)$ covers at least the same number of unit upward triangles as unit downward triangles of $\mfu(T)$. Thus, we would need at least two type-$2$ trapezoids to tile $T \setminus A(\mfc)$ into unit rhombi and unit trapezoids, which contradicts condition~(2).
	
	It is clear that $\mfc$ cannot be an independent set of $\T_n$; otherwise, we could use Theorem~\ref{thm:characterization of independent sets} to tile $T_n \setminus A(\mfc)$ using $0$ type-2 trapezoids, contradicting that the minimum number of type-2 trapezoids needed to tile $T_n \setminus A(\mfc)$ into unit rhombi and unit trapezoids is $1$. This, along with the fact that $|\mfu(T) \cap \mfc| \le \ell + 1$ for each lattice upward triangle $T$ of $T_n$ of size $\ell$, implies that $|\mfu(T') \cap \mfc| = \ell + 1$ for some lattice upward triangle $T'$ of $T_n$ of size $\ell$.
	
	We finally verify that $\mfc$ is a circuit of $\T_n$. Take $X \in \mfc$. By condition~(1), the only lattice upward triangle of $T_n$ strictly over-saturated by $\mfc$ is the triangular hull $T$ of $\mfc$. Set $\ell = \text{size}(T)$. The existence of a lattice upward triangle of $T_n$ over-saturated by $\mfc$ by exactly one unit upward triangle forces $|\mfu(T) \cap \mfc| = \ell + 1$. Since $T$ is the triangular hull of $\mfc$, it follows that $X \in \mfu(T)$ and, therefore, $T$ is not strictly over-saturated by $\mfc \setminus \{X\}$. Because no lattice upward triangle of $T_n$ is strictly over-saturated by $\mfc \setminus \{X\}$, the latter is an independent set of $\T_n$. As $X$ was arbitrarily chosen, $\mfc$ is a circuit.
\end{proof}

The following example illustrates that neither of the two conditions in Theorem~\ref{thm:characterization of circuit} by itself suffices to ensure that $\mfc$ is a circuit of $\T_n$.

\begin{example}
	Consider the tiling matroid $\T_4$. Let $\mfs \subset \mfu(T_4)$ consists of the dark unit upward triangles in the left picture of Figure~\ref{fig:circuit violation}. Note that $\mfs$ satisfies condition~(1) of Theorem~\ref{thm:characterization of circuit} since the only lattice triangle of $T_4$ that is strictly over-saturated by $\mfs$ is $T_4$, which is the triangular hull of $\mfs$. However, $\mfs$ is not a circuit of $\T_4$ (observe that we need at least two type-$2$ trapezoids to tile $T_4 \setminus A(\mfs)$). Now let $\mfs'$ be the set of dark unit upward triangles in the right picture of Figure~\ref{fig:circuit violation}. Observe that $\mfs'$ satisfies condition~(2) of Theorem~\ref{thm:characterization of circuit}; a tiling of $T_4 \setminus A(\mfs')$ as described in condition~(2) is shown in the picture. However, $\mfs'$ is not a circuit of $\T_4$ (one can see that the triangular hull of $\mfs'$ is not the only lattice upward triangle of $T_4$ strictly over-saturated by $\mfs'$).
	\begin{figure}[h]
		\begin{center}
			\includegraphics[width=4cm]{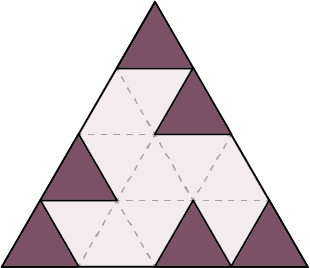}
			\hspace{2cm}
			\includegraphics[width=4cm]{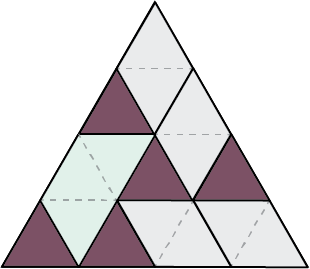}
			\caption{Two subsets of $\mfu(T_4)$ that are not circuits of $\T_4$.}
			\label{fig:circuit violation}
		\end{center}
	\end{figure}
\end{example}
\bigskip

\section{Geometric Characterization of the Flats} \label{sec:flats}

Let $M$ be a matroid with ground set $E$. The \emph{closure operator} $\text{cl} \colon 2^E \to 2^E$ of $M$ is defined as follows:
\[
	\text{cl}(S) = \{x \in E \mid r(S \cup \{x\}) = r(S)\}.
\]
Matroids can be characterized in terms of their closure operators; see \cite[Section~1.4]{jO92}. A subset $S$ of $E$ is called a \emph{flat} of $M$ provided that $\text{cl}(S) = S$. In this section, we give a geometric description of the flats of tiling matroids. To facilitate this, let us first introduce some notation.
 
Let $\mfs$ be a subset of $\mfu(T_n)$. Given two lattice upward triangles $T$ and $T'$ of $T_n$, we let $T \vee T'$ denote the triangular hull of $T \cup T'$. On the other hand, we say that the lattice upward triangle $T$ of $T_n$ of size $k$ is \emph{completely over-saturated} by $\mfs$ if $\mfu(T) \subseteq \mfs$. We use the following result in the proof of Proposition~\ref{prop:flats of tiling matroids}

\begin{lemma} \cite[Lemma~4.2]{AB07} \label{lem:on saturated simplices}
	Let $\mfs$ be an independent set of $\T_n$, and let $T$ and $T'$ be two lattice upward triangles of $T_n$ saturated by $\mfs$. If $T \cap T' \neq \emptyset$, then the lattice upward triangles $T \cap T'$ and $T \vee T'$ of $T_n$ are also saturated by $\mfs$.
\end{lemma}

\begin{prop} \label{prop:flats of tiling matroids}
	A subset $\mff$ of $\mfu(T_n)$ is a flat of $\T_n$ if and only if every lattice upward triangle of $T_n$ over-saturated by $\mff$ is also completely over-saturated by $\mff$.
\end{prop}

\begin{proof}
	Suppose first that $\mff$ is a flat of $\T_n$. It suffices to prove that every maximal lattice upward triangle of $T_n$ over-saturated by $\mff$ is completely over-saturated by $\mff$. Let $M_1, \dots, M_m$ be the maximal lattice upward triangles of $T_n$ over-saturated by $\mff$. It follows by Lemma~\ref{lem:on saturated simplices} that the $M_i$'s are pairwise disjoint. Therefore the independent subsets of $\mff$ are unions $\mfi_1 \cup \dots \cup \mfi_m$, where each $\mfi_j$ is an independent set of $\T_n$ contained in $\mfu(M_j)$. Fix $j \in [m]$, and let us check that $M_j$ is completely over-saturated by $\mff$. To do so, take $X \in \mfu(M_j)$. Note that for each independent set $\mfi$ contained in $\mff \cup \{X\}$,
	\[
		r(\mfi) = \sum_{i=1}^m r(\mfi \cap \mfu(M_i)) \le \sum_{i=1}^m r(\mfu(M_i)) \le \sum_{i=1}^m \text{size}(M_i) = r(\mff).
	\]
	Thus, $r(\mff \cup \{X\}) = r(\mff)$. Since $\mff$ is a flat of $\T_n$, it follows that $X \in \mff$. This implies that $\mfu(M_j) \subseteq \mff$. Hence $M_i$ is completely over-saturated by $\mff$ for every $i \in [m]$.

	For the backward implication, let $M_1, \dots, M_m$ be the maximal lattice upward triangles of $T_n$ that are over-saturated by $\mff$ (and, therefore, completely over-saturated by~$\mff$). By Lemma~\ref{lem:on saturated simplices}, the $M_i$'s are pairwise disjoint. For each $j \in [m]$, let $\mfi_j$ be an independent set of $\T_n$ contained in $\mfu(M_j)$. Take now $X \in \mfu(T_n) \setminus \mff$, and let us verify that $\mfs := \{X\} \cup \mfi_1 \cup \dots \cup \mfi_m$ is also an independent set of $\T_n$. Suppose, by contradiction, that $T$ is a lattice upward triangle that is strictly over-saturated by $\mfs$. As $\mfs \setminus \{X\} \subseteq \mff$, one has that $T$ is over-saturated by $\mff$. Clearly, $X \in \mfu(T)$. This implies that $T$ is a lattice upward triangle over-saturated by $\mff$ not contained in any of the $M_i$'s, which is a contradiction. Hence $\mfs$ is an independent set of $\T_n$ and, as a result,
	\[
		r(\mff \cup \{X\}) \ge r(\mfs) = 1 + \sum_{j=1}^m r(\mfi_j) = 1 + r(\mff) > r(\mff).
	\]
	Since $X$ was arbitrarily taken in $\mfu(T_n) \setminus \mff$, it follows that $\mff$ is a flat of the tiling matroid $\T_n$, which concludes the proof.
\end{proof}

\begin{cor}
	Each flat of $\T_n$ consists of all unit upward triangles in the union of a disjoint collection of lattice upward triangles.
\end{cor}

%A non-flat subset of $\mfu(T_n)$ might still consist of all unit upward triangles contained in the disjoint union of lattice upward triangles of sides greater than one. The right picture of Figure~\ref{ex:example of non-flat subsets} sheds some light upon this observation.

\begin{example} \label{ex:example of non-flat subsets}
	 Figure~\ref{fig:example of non-flat subsets} shows a flat (on the left) and a non-flat $\mfs$ (on the right). Even though $\mfs$ is not a flat, note that it consists of all the unit upward triangles inside a disjoint union of lattice upward triangles (the ones completely over-saturated by $\mfs$).
	\begin{figure}[h]
		\begin{center}
			\includegraphics[width=4cm]{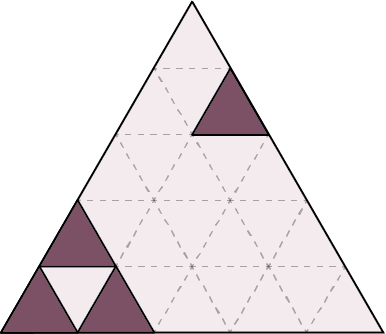}
			\hspace{2.5cm}
			\includegraphics[width=4.3cm]{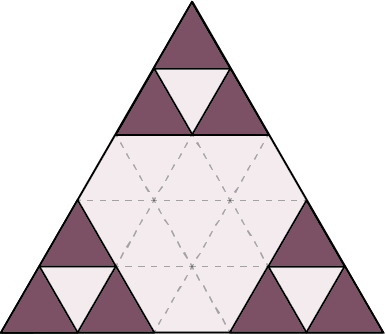}
			\caption{On the left, a flat of $\T_5$. On the right, a subset of $\mfu(T_5)$ that is not a flat of $\T_5$.}
			\label{fig:example of non-flat subsets}
		\end{center}
	\end{figure}
\end{example}

We conclude the paper with an open question for the reader. Since our main motivation comes from the set of lattice points of a regular simplex, the boundary regions of the matroids we have studied are equilateral triangles. However, it might be worthy to further investigate in the direction of the following question.

\begin{question}
	  For which more general regions allowing subdivisions into unit upward and downward triangles the results we have presented here can be extended?
\end{question}

\section*{Acknowledgements}

	While working on this paper, the first author was supported by the NSF-AGEP Fellowship. Both authors thank Federico Ardila for giving the initial motivation for this project and Lauren Williams for many useful suggestions.

\end{document}